\documentclass[a4paper,12pt, reqno]{amsart}
\usepackage{amsmath,amsfonts,amssymb,amsbsy,amsthm,epsfig,color,mathrsfs,nicefrac}
\usepackage{hyperref} 

\newcommand{\ignore}[1]{}

\newcommand{\bb}{\mathbb}

\newcommand{\C}{\bb C}

\newcommand{\Z}{\bb Z}
\newcommand{\height}{\textrm{\rm H}}

\newcommand{\R}{\bb R}
\newcommand{\N}{\bb N}
\newcommand{\p}{\bb P}

\newcommand{\W}{\mathcal W}
\newcommand{\Q}{\mathbb Q}

\newcommand{\bx}{\mathbf{x}}

\newcommand{\bxp}{\mathbf{x}'}

\newcommand{\by}{\mathbf{y}}
\newcommand{\byp}{\mathbf{y}'}

\newcommand{\q}{\mathbf{q}}

\newcommand{\cW}{\mathcal{W}_{v}(\psi, k, n)}
\newcommand{\cqW}{\mathcal{W}_{p}(\psi, \Q, n)}

\newcommand{\cR}{\mathcal{R}}

\newcommand{\rar}{\rightarrow}

\newtheorem{theorem}{Theorem}

\newtheorem{Prop}[theorem]{Proposition}

\newtheorem*{lemma*}{Lemma}

\newtheorem*{theorem*}{Theorem}

\numberwithin{equation}{section}
\numberwithin{theorem}{section}

\begin{document}
\title[Diophantine approximation]{Projective metric number theory}
\author{Anish Ghosh and Alan Haynes}
\thanks{AH supported by EPSRC grant EP/J00149X/1. AG supported by EPSRC}
\address{School of Mathematics, University of East Anglia, Norwich, UK }
\email{a.ghosh@uea.ac.uk}
\address{School of Mathematics, University of Bristol, Bristol UK }
\email{alan.haynes@bristol.ac.uk}

\begin{abstract}
In this paper we consider the probabilistic theory of Diophantine approximation in projective space over a completion of $\Q$. Using the projective metric studied in \cite{BVV} we prove the analogue of Khintchine's Theorem in projective space. For finite places and in higher dimension, we are able to completely remove the condition of monotonicity and establish the analogue of the Duffin-Schaeffer conjecture.
\end{abstract}

\maketitle


\section{Introduction}

The subject of \emph{metric Diophantine approximation} is concerned with estimating the size of sets with prescribed Diophantine properties. A foundational theorem is due to Khintchine. Let $\psi : \R_{+} \cup \{0\} \to \R_{+} \cup \{0\}$ be a decreasing function. Then the set of ``$\psi$-approximable numbers," namely those for which the inequality
$$ |x - p/q| < \psi(|q|) $$
\noindent holds for infinitely many $p, q \in \Z$, has zero or full measure according to whether
$$ \sum_{q = 1}^{\infty} q\psi(q) $$
\noindent converges or diverges. The result generalizes naturally to higher dimensions. Probably the most important open problem in classical metric Diophantine approximation is the Duffin-Schaeffer conjecture (\cite{DS}) which asks to what extent the monotonicity condition on $\psi$ can be relaxed. It is known by work of Pollington and Vaughan \cite{PV} that in dimensions greater than $1$, monotonicity is not essential, i.e. the higher dimensional Duffin-Schaeffer conjecture is true. The problem in dimension $1$ is more delicate and has been studied in several recent works (cf.  \cite{HPV}, and \cite{Haynes}).\\

The purpose of this paper is to investigate projective analogues of these theorems. First we establish the projective Duffin-Schaeffer conjecture in higher dimensions for finite places of $\Q$. Subsequently we show how to use existing results on the distribution of rational points of bounded height in projective space, coupled with ubiquitous systems, to establish a projective analogue of Khintchine's theorem in any dimension and for both finite and infinite places. The motivation for our work comes from \cite{CV} where K. Choi and J. Vaaler established a projective version of Dirichlet's theorem in Diophantine approximation for number fields.\\

Let $k$ be a number field and $k_v$ its completion at the place $v$. Let $\|~\|$ be an absolute value from $v$ which extends the Euclidean absolute value on $k_v$ if $v~|~\infty$ and the $p$-adic absolute value if $v | p$. If $d = [k : \Q]$ and $d_v = [k_v: \Q_v]$ are the global and local degrees of the extension then for $x \in k_v$ define the normalized absolute value
$$ |x|_v  := \| x \|_{v}^{d_v/d}.$$
\noindent The normalized absolute values then satisfy the usual product formula. Next for $n\in\N$ and $\bx \in k_{v}^{n}$ define
\begin{equation*} \|\bx\|_{p} := \left\{
\begin{array}{rl} \max\{\|x_i\|_v, 1 \leq i \leq n\} & \text{if } v\text{ lies over a finite prime},\\\\ \left(\sum_{i = 1}^{n} \|x_i\|_v^2 \right)^{1/2} &\text{if } v | \infty,
\end{array} \right. \end{equation*}
and let $|\bx|_v := \|\bx\|_{v}^{d_v/d}$.\\

Projective metric Diophantine approximation aims to quantify the density of $\p^{n-1}(k)$ in $\p^{n-1}(k_v)$. For this we need a metric and a height function. For non-zero vectors $\bx, \by \in k^{n}_v$ we define
\begin{equation}\label{def-metric}
\delta_{v}(\bx, \by) := \frac{|\bx \wedge \by|_v}{|\bx|_v |\by|_v}.
\end{equation}
Then $\delta_v$ defines a metric on $\p^{n-1}(k_v)$ which induces the usual quotient topology (\cite{Rumely}). For example it is easy to check for $k=\Q$ and $v=\infty$ that $\delta_{\infty}(\bx, \by)  = |\sin(\theta)|$ where $\theta$ is the angle between $\bx$ and $\by$. We define the height of a point $\bx\in \p^{n-1}(k)$ by
\begin{equation}\label{height-def}
\height(\bx) := \prod_{v}|\bx|_v,
\end{equation}
and we note that this is well defined over projective space because of the product formula. We are now ready to state Choi and Vaaler's projective version of Dirichlet's theorem (Theorem $1$ in \cite{CV}).
\begin{theorem}\label{Choi-Vaaler}
Let $\bx \in \p^{n-1}(k_v)$, let $\tau \in k_v$ with $|\tau|_v \geq 1$. Then there exists $\by \in \p^{n-1}(k)$ such that
\begin{enumerate}
\item $\height(\by) \leq c_{k}(n)|\tau|_{v}^{n-1}$,\text{ and}\\
\item $\delta_{v}(\bx, \by) \leq c_{k}(n)(|\tau|_v\height(\by))^{-1}$.
\end{enumerate}
\end{theorem}
\noindent Here
\begin{align}
c_{k}(n) = 2|\Delta_{k}|^{1/2d} \prod_{v \arrowvert \infty} r_{v}(n)^{d_v/d},\label{constant}
\end{align}
$\Delta_k$ is the discriminant of $k$, and
\begin{align*}
r_{v}(n) = \left\{
\begin{array}{rl}  \pi^{-1/2}\Gamma(\frac{n}{2} + 1)^{1/n} & \text{if } v \text{ is real},\\\\ (2\pi)^{-1/2}\Gamma(n + 1)^{1/2n} & \text{if } v \text{ is complex}.
\end{array} \right.
\end{align*}
A corollary of Theorem \ref{Choi-Vaaler}, is that for every $\bx \in \p^{n-1}(k_v) \backslash \p^{n-1}(k)$, there exist infinitely many distinct $\by \in \p^{n-1}(k)$ such that
$$\delta_v(\bx, \by) \leq c_{k}(n)^{n/(n-1)}\height(\by)^{-n/(n-1)}.$$\\

In order to state our results we will work with probability measures on $\p^{n-1}(k_v)$, originally defined and studied by Choi \cite{Choi}. First we specify natural measures $\beta_v^n$ on $k_v^n$. If $v$ is an infinite place then $\beta_v^n$ is the usual $n$-fold Lebesgue measure on $\R^n$ or $\C^n$, while if $v$ is a finite place then $\beta_v^n$ is the $n-$fold Haar measure normalized so that
$$ \beta_v(O_v) = \|\mathcal{D}_v\|_{v}^{d_v/2}, $$
where $O_v$ is the ring of integers of $k_v$ and $\mathcal{D}_v$ is the local different of $k$ at $v$. Next let $\phi :k_v^n\setminus \{{\bf 0}\}\rar\p^{n-1}(k_v)$ be the quotient map and define the $\sigma$-algebra $\mathcal{M}$ of measurable sets in $\p^{n-1}(k_v)$ to be the collection of sets $M\subseteq\p^{n-1}(k_v)$ such that $\phi^{-1}(M)$ lies in the $\sigma-$algebra of Borel sets in $k_v^n$. Then define measures $\mu_v$ on $(\p^{n-1}(k_v),\mathcal{M})$ by
\[\mu_v(M)=\frac{\beta_v^n\left(\phi^{-1}(M)\cap B({\bf 0},1)\right)}{\beta_v^n\left(B({\bf 0},1)\right)}.\]\\

To simplify the exposition from here on we will specialize to the case when $k = \Q$. The case of a general number field is still interesting but more technical, and we leave its treatment to a later paper. Given $\psi : \R_{+} \cup \{0\} \to \R_{+} \cup \{0\}$ let $\cW$ be the set of $\bx \in \p^{n-1}(\Q_v)$ for which there exist infinitely many $\by \in \p^{n-1}(\Q)$ such that
$$ \delta_{v}(\bx, \by) \leq \psi(\height(\by)).$$



\noindent Then it is a straightforward consequence of the Borel-Cantelli lemma that $\mu_{p}(\cW) = 0$ whenever
\begin{equation}\label{sum}
\sum_{q=1}^{\infty} q^{n - 1}\psi(q)^{(n - 1)}
\end{equation}
\noindent converges. In particular this implies that the power $-n/(n-1)$ in Theorem \ref{Choi-Vaaler} is generically best possible. Our first result establishes the projective $p$-adic version of the Duffin-Schaeffer conjecture in all dimensions greater than $1$.

\begin{theorem}\label{thm2}
Assume that $p$ is a finite place, that $n > 2$, and that $\psi (q)=0$ whenever $p|q$. Then $\cqW$ has full measure whenever (\ref{sum}) diverges.
\end{theorem}

\noindent Now we allow arbitrary primes and dimensions at the cost of monotonicity. Our second result is therefore the complete projective version of Khintchine's theorem.

\begin{theorem}\label{thm1}
Assume that $\psi$ is decreasing and let $p$ be a (finite or infinite) place of $\Q$. Then $\mu_{p}(\cqW)=1$ whenever (\ref{sum})
diverges.
\end{theorem}

\noindent We note that it is not difficult to show that the monotonicity assumption in Theorem \ref{thm1} can be removed when $n=2$ if and only if the Duffin-Schaeffer conjecture is true (see \cite{DS}, \cite{HPV}, and \cite{Haynes} for more details concerning this conjecture). We will demonstrate in \S\ref{0-1law subsec} that the condition that $\psi(q)=0$ whenever $p|q$ turns out to be a natural one, since without it there is not even a zero-one law in general. In other words without this condition it is possible to choose $\psi$ so that $0<\mu_p(\W_p(\psi))<1.$ Our method could easily be extended to deal with this more general case without introducing any new ideas, but for simplicity and elegance of the proofs we impose the extra condition.\\

\noindent Finally we remark that the version of Theorem \ref{thm2} for infinite places essentially follows from Gallagher's proof of \cite[Theorem 1]{Gallagher2}, albeit with some modifications. Also Theorems \ref{thm2} and \ref{thm1}, when combined with the Mass Transference Principle of Beresnevich and Velani \cite{BVel}, yield Hausdorff measure and dimension analogues of Khintchine's theorem. Since the application of mass transference is for the most part straightforward we omit the details. The interested reader can see \cite[Section 6]{Haynes} for an example of how the Hausdorff dimension arguments proceed.

\subsection*{Acknowledgements} AG thanks the ESI, Vienna for hospitality. AH thanks Simon Kristensen for helpful conversations concerning the proof of Theorem \ref{0-1law thm}.

\section{Projective Duffin-Schaeffer Conjecture}\label{0-1law sec}
\subsection{Basic setup}\label{p-adic setup subsec}
First we demonstrate the proof of Theorem \ref{thm2}. If $\psi(q)\ge 1$ for infinitely many $q$ then the statement of the theorem is trivial to verify, since $\delta_p(\bx, \by)\le 1$ for all $\bx,\by\in\p^{n-1}(\Q_p)$. Therefore by a straightforward application of the Borel-Cantelli lemma we may restrict our attention to the situation when $\psi$ takes values only in the set $\{0\}\cup\{p^{-k}:k\in\N\}$.

For $1\le i\le n$ define $E_i\subseteq\p^{n-1}(\Q_p)$ by
\[E_i:=\phi\left(\left\{\bx\in\Z_p^n~:~|x_i|_p\ge |x_j|_p \text{ for } 1\le j\le n\right\}\right),\]
and for each $q\in\N$ define $A_{q,i}(\psi)\subseteq\p^{n-1}(\Q_p)$ by
\begin{equation}\label{A_q,i def}
A_{q,i}(\psi):=\bigcup_{\substack{\by\in\Z^n_{\mathrm{vis}}\\\height (\by)=y_i=q}}B_\delta(\phi(\by),\psi (q)),
\end{equation}
where
\[\Z_\mathrm{vis}^n=\{\by\in\Z^n:\gcd (y_1,\ldots ,y_n)=1\}.\]
If we let
\[\W_{p,i}(\psi):=\limsup_{q\rar\infty}A_{q,i}(\psi)\]
then we have that
\begin{equation}\label{limsup union eqn}
\W_p(\psi)=\bigcup_{i=1}^n\W_{p,i}(\psi).
\end{equation}
A basic observation which is useful to us is the following characterization of balls in $\p^{n-1}(\Q_p)$.
\begin{Prop}\label{ball prop}
For any $k\in\N$ and $\by\in\Z_p^n$ with $|\by|_p=1,$
\begin{equation}\label{ball eqn1}
B_\delta(\phi(\by) ,p^{-k})=\phi\left(\{\bx\in\Z_p^n:|\bx-\by|_p\le p^{-k}\}\right).
\end{equation}
\end{Prop}
\begin{proof}
First suppose that $\bx\in\Z_p^n$ satisfies $|\bx-\by|_p\le p^{-k}$. Then it follows that $|\bx|_p=1$ and
    \[\delta_p(\phi(\bx) ,\phi(\by))=\max_{1\le j<j'\le n}|x_jy_{j'}-x_{j'}y_j|_p\le p^{-k}.\]
For the other direction suppose that $\bxp\in B_\delta(\phi(\by) ,p^{-k})$ and choose $\bx\in\phi^{-1}(\bxp)$ with $|\bx|_p=1$. Since $|\by |_p=1$ we can choose $1\le i\le n$ so that $|y_i|_p=1$. We cannot have that $x_i=0\mod p$ since (\ref{ball eqn1}) would then imply that
    \[x_j=0\mod p~\text{for all }~1\le j\le n,\]
    contradicting our assumption that $|\bx|_p=1.$ Therefore, by multiplying $\bx$ by $y_ix_i^{-1}$, we can assume that $\bx\in\phi^{-1}(\bxp)$ has been chosen so that $|\bx|_p=1$ and $x_i=y_i$. Then (\ref{ball eqn1}) implies that $|\bx-\by|_p\le p^{-k}.$
\end{proof}
An immediate consequence of this is that if $k\in\N,$ $\by\in\Z_p^n,$ and $|\by|_p=|y_i|_p=1$ then $B_\delta(\phi(\by),p^{-k})\subseteq E_i.$ Therefore, under the assumption that $\psi(q)=0$ whenever $p|q$, we have for each $i$ and $q$ that $A_{q,i}(\psi)\subseteq E_i$ and thus $\W_{p,i}(\psi)\subseteq E_i$.

\subsection{Zero-one law}\label{0-1law subsec}
First we will show that there is not a zero-one law without the assumption that $\psi (q)=0$ whenever $p|q$.
\begin{Prop} For the function
\[\psi (q)=
\begin{cases}
    p^{-1}&\text{if}~p|q,\\
    0&\text{otherwise},
\end{cases}    \]
we have that
\[\W_p(\psi)=\phi\left(\left\{\bx\in\Z_p^n:|\bx|_p=1\text{ and } \min_{1\le i\le n}|x_i|_p\le p^{-1}\right\}\right).\]
Therefore in this case
\[\mu_p(\W_p(\psi))=1-\frac{(p-1)^n}{p^n-1}.\]
\end{Prop}
\begin{proof}
First suppose that $\by\in\Z^n_\mathrm{vis}$ has $\height (\by )=y_i$ and $p|y_i$. Then if $\bx\in\Z_p^n$ satisfies $|\bx-\by|_p\le p^{-1}$ we must have $|\bx|_p=1$ and $|x_i|_p\le p^{-1}.$ By Proposition \ref{ball prop} this shows that
\[\W_p(\psi)\subseteq\phi\left(\left\{\bx\in\Z_p^n:|\bx|_p=1\text{ and } \min_{1\le i\le n}|x_i|_p\le p^{-1}\right\}\right).\]
For the other inclusion suppose that $\bx\in\Z_p^n$ satisfies $|\bx|_p=1$ and $|x_i|_p\le p^{-1}$ for some $1\le i\le n$. For any $q\in\N$ with $p|q$ we can choose a point $\by\in\Z^n_\mathrm{vis}$ with $\height (\by)=q$ and $\phi(\bx)\in B(\phi(\by),p^{-1}),$ by requiring that $y_i=q$ and for $j\not=i$ that $y_j$ be the least non-negative representative for $x_j\mod p$. Since $q\in p\N$ is arbitrary it follows that $\phi(\bx)\in\W_p(\psi)$.

For the measure calculation we have that
\[\W_p(\psi)^c=\phi\left(\{\bx\in\Z_p^n:|x_1|_p=\cdots =|x_n|_p=1\}\right)\]
and therefore
\begin{align*}
\mu_p(\W_p(\psi))&=1-\beta_p^n\left(\phi^{-1}(\W_p(\psi)^c)\right)\\
&=1-\sum_{i=0}^\infty\beta_p^n(\{\bx\in\Z_p^n : |x_1|_p=\cdots =|x_n|_p=p^{-i}\})\\
&=1-\sum_{i=0}^\infty\left(\frac{p-1}{p^{i+1}}\right)^n\\
&=1-\frac{(p-1)^n}{p^n-1}.
\end{align*}
\end{proof}
In light of this example we will assume in all of what follows that $\psi (q)=0$ whenever $p|q$. In this case we have the following theorem.
\begin{theorem}\label{0-1law thm}
    For any choice of $\psi$ satisfying $\psi(q)=0$ whenever $p|q$ we have that $\mu_p(\W_p(\psi))=0$ or $1$.
\end{theorem}
In our proof we will reduce the problem to a problem about limsup sets in $\Z_p^{n-1}$ by using the bijective maps $\eta_i:\Z_p^{n-1}\rar E_i, ~1\le i\le n,$ defined by \[\eta_i(x_1,\ldots ,x_{n-1})=(x_1,\ldots ,x_{i-1},1,x_i,\ldots ,x_{n-1}).\] This will position us to apply a known zero-one law from \cite{Haynes}, from which Theorem \ref{0-1law thm} will follow from the following proposition.
\begin{Prop}\label{eta measure prop}
If $M\subseteq\Z_p^{n-1}$ is measurable then for $1\le i\le n$,
\[\mu_p\left(\eta_i(M)\right)=\frac{p^n-p^{n-1}}{p^n-1}\cdot\beta_p^{n-1}(M).\]
\end{Prop}
\begin{proof}
For notational convenience let us simply prove the case when $i=n$. Notice that
\begin{align*}
\phi^{-1}(\eta_n(M))\cap\Z_p^n&=\{z(\bx,1):\bx\in M, z\in\Z_p\}\\
&=\bigcup_{\ell=0}^\infty p^\ell\cdot\{u(\bx,1):\bx\in M, u\in U_p\},
\end{align*}
where $U_p$ denotes the group of units in $\Z_p$. Therefore we have
\begin{align*}
\mu_p\left(\eta_n(M)\right)&=\sum_{\ell=0}^\infty p^{-n\ell}\cdot\beta_p^n\left(\{u(\bx,1):\bx\in M, u\in U_p\}\right)\\
&=\sum_{\ell=0}^\infty p^{-n\ell}\int_{U_p}\int_{\Z_p^{n-1}}\chi_{u^{-1}M}(\bx)~d\beta_p^{n-1}(\bx)d\beta_p(u)\\
&=\sum_{\ell=0}^\infty p^{-n\ell}\int_{U_p}\beta_p^{n-1}(u^{-1}M)~d\beta_p(u).
\end{align*}
For any $u\in U_p$ we have that $\beta_p^{n-1}(u^{-1}M)=\beta_p^{n-1}(M)$, which gives
\begin{align*}
\mu_p\left(\eta_n(M)\right)&=\beta_p^{n-1}(M)\beta_p(U_p)\sum_{\ell=0}^\infty p^{-n\ell}\\
&=\frac{p^n-p^{n-1}}{p^n-1}\cdot\beta_p^{n-1}(M).
\end{align*}
\end{proof}
\begin{proof}[Proof of Theorem \ref{0-1law thm}]
Recall that for each $i$ we have $\W_{p,i}(\psi)\subseteq E_i$. It is also clear by the symmetry of the definitions that $\mu_p(\W_{p,i}(\psi))=\mu_p(\W_{p,j}(\psi))$ for all $i,j$. We will now show that, under the hypotheses of our theorem,
\begin{equation}\label{0-1law in Z_p}
\beta_p^{n-1}(\eta_n^{-1}(\W_{p,n}(\psi)))=0\text{ or }1.
\end{equation}
To see this, suppose that $0<\psi (q)<1$ and that $\by\in\Z^n_\mathrm{vis}$ has $\height (\by)=y_n=q$. Then it follows from Proposition \ref{ball prop} that
\begin{align}
&\eta_n^{-1}(B_\delta(\phi (\by),\psi (q)))\nonumber\\
&\qquad =\left\{\bx\in\Z_p^{n-1}:\left|x_i-\frac{y_i}{q}\right|_p\le\psi (q)\text{ for }1\le i\le n-1\right\}.\label{ball pullback eqn}
\end{align}
Therefore a point $\bx\in\Z_p^{n-1}$ belongs to $\eta_n^{-1}(W_{p,n}(\psi))$ if and only if
\[\max_{1\le i\le n-1}\left|x_i-\frac{y_i}{q}\right|_p\le\psi (q)\]
for infinitely many $q\in\N$ and $(y_1,\ldots ,y_{n-1})\in\Z^{n-1}$ satisfying $|y_i|\le q$ for each $i$ and $\gcd (y_1,\ldots ,y_{n-1},q)=1$. By a minor modification of the proof of \cite[Lemma 1]{Haynes} we conclude that (\ref{0-1law in Z_p}) holds.

Remark: To alleviate any doubt about this last sentence, the only modification necessary in using \cite[Lemma 1]{Haynes} is to justify that the proof of that lemma still works with the difference in our gcd conditions (i.e. in \cite{Haynes} it deals with the case where $\gcd(y_i,q)=1$ for all $i$). However the gcd condition only comes up in one place in the proof, and it causes no problems for our setup.

Finally we will combine (\ref{0-1law in Z_p}) with Proposition \ref{eta measure prop}. On one hand if $\beta_p^{n-1}(\eta_n^{-1}(\W_{p,n}(\psi)))=0$ then we have that $\mu_p(\W_{p,i}(\psi))=0$ for all $i$ and (\ref{limsup union eqn}) gives that $\mu_p(\W_p(\psi))=0$.

On the other hand if $\beta_p^{n-1}(\eta_n^{-1}(\W_{p,n}(\psi)))=1$ then for each $i$ we have that
\[\mu_p(\W_{p,i}(\psi))=\frac{p^n-p^{n-1}}{p^n-1}=\mu_p(E_i).\]
Then since
\begin{align*}
\W_p(\psi)^c=\bigcup_{i=1}^n\left(\W_{p,i}^c\cap E_i\right)
\end{align*}
it follows that
\[\mu_p\left(\W_p(\psi)^c\right)\le\sum_{i=1}^n\mu_p\left(\W_{p,i}^c\cap E_i\right)=0,\]
and $\mu_p(\W_p(\psi))=1.$
\end{proof}

\subsection{Proof of Theorem \ref{thm2}}\label{p-adic proof subsec}
We remind the reader that we are assuming that $n>2$ and that $\psi(q)$ takes values only in the set $\{0\}\cup\{p^{-k}:k\in\N\}$, with $\psi (q)=0$ whenever $p|q$. First we point out that the divergence of the sum (\ref{sum}) is equivalent to the divergence of
\begin{equation}\label{meas div eqn}
\sum_{q=1}^\infty\mu_p (A_{q,n}(\psi)).
\end{equation}
To see this note that if $\psi (q)\le q^{-1}$ then by Proposition \ref{ball prop} the right hand side of (\ref{A_q,i def}) is a disjoint union. Therefore in this case
\begin{align*}
\mu_p(A_{q,n}(\psi))=\mu_p\left( B_\delta({\bf 1},\psi (q))\right)\cdot\#\{\by\in\Z^n_{\mathrm{vis}}:\height (\by)=y_n=q\}
\end{align*}
Now it is easy to verify that
\[\mu_p\left( B_\delta({\bf 1},\psi (q))\right)=\frac{\psi(q)^{n-1}}{1-p^{-n}}\]
and
\begin{align*}
\#\{\by\in\Z^n_{\mathrm{vis}}:\height (\by)=y_n=q\}&=\sum_{\substack{\byp\in\Z^{n-1}\\ 0<\|\byp\|_\infty\le q}}\sum_{d|q,y_1,\ldots ,y_{n-1}}\mu (d)\\
&=\sum_{d|q}\mu (d)\sum_{\substack{\byp\in\Z^{n-1}\\ 0<\|\byp\|_\infty\le q/d}}1\\
&=(2q)^{n-1}\sum_{d|q}\frac{\mu (d)}{d^{n-1}}\\
&\asymp q^{n-1},
\end{align*}
where $\|\cdot\|_\infty$ denotes the sup norm. Thus we have
\[\mu_p(A_{q,n}(\psi))\asymp q^{n-1}\psi (q)^{n-1}\]
for all $q$ with $\psi (q)\le q^{-1},$ and it is clear from this that (\ref{sum}) diverges if and only if (\ref{meas div eqn}) does. Furthermore if we define $\psi'$ by
\[\psi'(q)=\begin{cases}\psi (q)&\text{ if }\psi (q)\le q^{-1},\\q^{-1}&\text{otherwise},\end{cases}\]
then the convergence or divergence of (\ref{sum}) is the same with $\psi$ replaced by $\psi'$. Since $\W_p (\psi')\subseteq \W_p (\psi)$, this shows that it is sufficient to prove Theorem \ref{thm2} under the additional hypothesis that $\psi (q)\le q^{-1}$ for all $q$. For simplicity we make this assumption for the rest of the proof.

Now suppose that $q,r\in\N$ are distinct and that $\psi(q)\le\psi (r)=p^{-k}$. Then we have the upper bound
\begin{align*}
\mu_p (A_{q,n}(\psi)\cap A_{r,n}(\psi))\le \mu_p\left( B_\delta({\bf 1},\psi (q))\right)\cdot S(q,r),
\end{align*}
with
\begin{align*}
S(q,r)=\#\{\bx,\by\in\Z^n_{\mathrm{vis}}:&~\height (\bx)=x_n=q, \height (\by)=y_n=r,\\
&\qquad\delta_p (\bx,\by)\le\psi (r)\}.
\end{align*}
By Proposition \ref{ball prop} we obtain
\begin{align*}
S(q,r)\le\sum_{\substack{\bxp\in\Z^{n-1}\\ 0<\|\bxp\|_\infty\le q}}\sum_{\substack{\byp\in\Z^{n-1}\\ 0<\|\byp\|_\infty\le r\\q\byp-r\bxp =0\mod p^k}}1,
\end{align*}
where we have used the fact that $\gcd(x_1,...,x_n)=\gcd(y_1,...,y_n)=1$ to exclude the terms in the sums corresponding to $\bxp$ or $\byp={\bf 0}$. This guarantees that the sums are empty whenever $p^k>2qr$, therefore we can safely say that
\[S(q,r)\ll\left(\frac{qr}{p^k}\right)^{n-1}.\]
Substituting this above gives
\begin{align*}
\mu_p (A_{q,n}\cap A_{r,n})\ll (\psi(q)\psi(r)qr)^{n-1}\ll\mu_p (A_{q,n})\mu_p(A_{r,n}).
\end{align*}
Finally by a standard variance argument (e.g. \cite[Lemma 2.3]{Harman}) we have under the divergence of (\ref{meas div eqn}) that
\begin{align*}
\mu_p(\W_{p,n})\ge\limsup_{Q\rar\infty}\left(\sum_{q=1}^Q\mu_p (A_{q,n})\right)^2\left(\sum_{q,r=1}^Q\mu_p (A_{q,n}\cap A_{r,n})\right)^{-1}>0,
\end{align*}
and by Theorem \ref{0-1law thm} it follows that $\mu_p(\W_p(\psi))=1$.

\section{Projective Khintchine Theorem}\label{ubiquity}

\noindent In the proof of Theorem \ref{thm1}, we will use the notion of \emph{ubiquitous systems}. Ubiquitous systems are a modern avatar of \emph{regular systems} which originated in the work of A. Baker and W. M. Schmidt \cite{BS}. Subsequently, they were developed by Dodson, Rynne and Vickers \cite{DRV} and others and have proved to be a valuable tool in investigating problems in metric Diophantine approximation. We refer the reader to the work of Berenevich, Dickinson, and Velani \cite{BDV} for a very readable account of the history,  a modern, improved version which we will use, as well as a wealth of applications. We begin with the definition of a ubiquitous system, following \S 2 of \cite{BDV} and using their notation. Let $(\Omega, d)$ be a compact metric space equipped with a probability measure $m$ on the corresponding Borel $\sigma-$algebra.  Let $\cR = \{R_{\alpha}~:~\alpha \in J\}$ be a family of subsets $R_{\alpha}$ (called resonant sets) of $\Omega$ indexed by an infinite countable set $J$. Let $\beta : J \to \R_{+}$ and for $\delta > 0$ and $A \subset \Omega$ define $\Delta(A, \delta)$ to be the $\delta$ neighborhood of $A$,
$$\Delta(A, \delta) := \{x \in \Omega~:~d(x, A) < \delta\}.$$
\noindent Let $\rho : \R_{+} \to \R_{+}$ denote a function such that $\lim_{r \to \infty} \rho(r) = 0$, and let $l = \{l_n\}$ and $u = \{u_n\}$ be positive increasing sequences such that $l_n < u_n$ for all $n$ and $\lim_{n \to \infty}l_n = \lim_{n \to \infty} u_n = \infty$. Define
$$ \Delta^{u}_{l}(\rho, n) := \bigcup_{\alpha \in J^{u}_{l}(n)} \Delta(R_{\alpha}, \rho(\beta_{\alpha})),$$
where
$$J^{u}_{l} := \{\alpha \in J~:~l_n < \beta_{\alpha} \leq u_n \}. $$

\noindent We assume that the cardinality of $J^{u}_{l}(n)$ is finite for every $n$ and denote by $\Lambda(\rho)$ the set $\limsup_{n \to \infty} \Delta^{u}_{l}(\rho, n)$.

We say that $(\cR, \beta)$ is locally $m$-ubiquitous relative to $(\rho, l, u)$ if both of the following conditions are satisfied:
\begin{enumerate}
\item There exist $r_0, \kappa > 0$ such that for any $r \leq r_0$ and any ball $B = B(x, r)$,
\begin{equation}
m(B \cap \Delta_{l}^{u}(\rho, n)) \geq \kappa m(B)~\text{for}~n \geq n_{0}(B).
\end{equation}
\item There exists constants $0 \leq \gamma \leq \dim \Omega$ and $0 < c_1 < 1 < c_2$, such that for any $\alpha \in J$ with $\beta_{\alpha} \leq u_n, c \in R_{\alpha},$ $0 < \lambda \leq \rho(u_n)$ and large enough $n$:
\begin{enumerate}
\item $m(B(c, \frac{1}{2}\rho(u_n)) \cap \Delta(R_{\alpha}, \lambda)) \geq c_1m(B(c, \lambda))\left(\frac{\rho(u_n)}{\lambda}\right)^{\gamma}.$\\
\item $m(B \cap B(c, 3\rho(u_n)) \cap \Delta(R_{\alpha}, 3\lambda)) \leq c_2 m(B(c, \lambda))\left(\frac{r(B)}{\lambda}\right)^{\gamma},$
\end{enumerate}
\noindent where $r(B)$ denotes the radius of $B$.
\end{enumerate}
We will refer to the conditions in (2) of this definition as the \emph{intersection conditions}. We say that $(\cR, \beta)$ is globally $m$-ubiquitous relative to $(\rho, l, u)$ if both of the above conditions are satisfied with $B = \Omega$.

We need two more notions from \cite{BDV}. A measure $m$ on $\Omega$ is said to satisfy condition $(M2)$ if there exist positive constants $\delta, r_0, a, b$ such that for any $x \in \Omega$ and $r \leq r_0$
$$ ar^{\delta} \leq m(B(x, r)) \leq br^\delta. $$

\noindent A function $f$ is called $u$-regular for a sequence $u$ as above, if there exists a positive constant $\lambda < 1$ for which
$$ f(u_{n+1}) \leq \lambda f(u_n).$$
The following is \cite[Corollary 2]{BDV}.
\begin{theorem}\label{BDV1}
Let $\Omega$ be a compact metric space equipped with a probability measure $m$ satisfying condition $(M2)$. Suppose that $(\cR, \beta)$ is a globally $m$-ubiquitous system relative to $(\rho, l, u)$ and that $\psi$ is a decreasing function. Assume that $\delta > \gamma$, that either $\psi$ or $\rho$ is $u$-regular and that
\begin{equation}\label{BDV1-1}
\sum_{n = 1}^{\infty} \left(\frac{\psi(u_n)}{\rho(u_n)}\right)^{\delta - \gamma} = \infty.
\end{equation}
\noindent Then $m(\Lambda(\psi)) > 0$. In addition, if any open subset of $\Omega$ is measurable and $(\cR, \beta)$ is locally $m$-ubiquitous relative to $(\rho, l, u)$, then $m(\Lambda(\psi)) = 1$.
\end{theorem}

\noindent Now suppose that $p$ is a finite or infinite prime and let $\Omega = \p^{n-1}(\Q_p)$ and $J = \p^{n-1}(\Q)$. For $\alpha\in J$ let $\beta_{\alpha} := \height(\alpha)$ and $R_{\alpha} = \alpha$, so the resonant sets are rational points. Take $l_{i + 1} = u_i = 2^{i}$ and $\rho(r) = r^{-n/n-1}$. Let $I$ be a ball in $\p^{n-1}(\Q_p)$. Note that since the function $\psi$ in Theorem \ref{thm1} is decreasing, it is $u$-regular. Moreover the choice of $\rho$ means that if we take $\gamma = 0$ and $\delta = n-1$ then (\ref{BDV1-1}) coincides with (\ref{sum}). Therefore to prove Theorem \ref{thm1} it is sufficient to establish the following proposition.
\begin{Prop}\label{ubiquitous}
The system $(\mathcal{R}, \beta)$ is locally $\mu_p$-ubiquitous with respect to $(\rho, l, u)$.
\end{Prop}

\begin{proof}
First note that the intersection conditions are satisfied with $\gamma = 0$. By Theorem \ref{Choi-Vaaler} after rescaling the metric appropriately, we have that for every $\bx \in I$ and $Q \geq 1$ there exists $\by \in \p^{n-1}(\Q)$ such that
\begin{equation}\label{ubi1}
\delta_{p}(\bx, \by) \leq \frac{1}{Q\height(\by)}
\end{equation}

\noindent and

\begin{equation}\label{ubi2}
\height(\by) \leq Q^{n-1}.
\end{equation}

\noindent We set $Q = u_i$ and use the fact that the measure of a ball of radius $r$ in $\p^{n-1}(\Q_p)$ is $\asymp r^{(n-1)}$. We will also need results of Schanuel (\cite{Schanuel}) and Choi (\cite{Choi}) on the distribution of rational points. For every $T > 0$, there are finitely many $\by \in I \cap \p^{n-1}(\Q)$ with $\height(\by) \leq T$ and it an important problem in arithmetic geometry to understand the behaviour of the counting function as $T \to \infty$. From loc. cit. we have that
\begin{equation}\label{count}
\#\{\by \in I \cap \p^{n-1}(\Q) ~:~\height(\by) \leq T\} \asymp T^{n}\mu_{p}(I)
\end{equation}

\noindent We can now calculate
\begin{align*}
&\mu_{p}\left(I \cap \bigcup_{\substack{\by \in I\\
\height(\by) \leq 2^{i(n-1)-1}}} B\left(\by, \frac{1}{\height(\by)2^{i}}\right) \right)\\
&\qquad\qquad\leq \sum_{\substack{\by \in I\\
\height(\by) \leq 2^{i(n-1)-1}}} \left(\frac{1}{\height(\by)2^{i}}\right)^{(n-1)}\\
&\qquad\qquad\leq \frac{1}{2^{i(n-1)}} \sum_{\substack{\by \in I\\
\height(\by) \leq  2^{i(n-1)-1}}}\frac{1}{\height(\by)^{(n-1)}}\\
&\qquad\qquad\leq \frac{\mu_{p}(I)}{2^{i(n-1)}} \sum_{j = 1}^{i(n-1)-1}\left(\frac{1}{2^j}\right)^{n-1}2^{jn}\left(1 - \frac{1}{2^n}\right)\\
&\qquad\qquad\leq \frac{\mu_{p}(I)}{2^{i(n-1)}}\left(1 - \frac{1}{2^n}\right)  \sum_{j = 1}^{i(n-1)-1}2^j\\
&\qquad\qquad\leq \left(1 - \frac{1}{2^{n}}\right)\frac{\mu_{p}(I)}{2^{i(n-1)}}2(2^{i(n-1)-1} - 1)\\
&\qquad\qquad\leq \left(1 - \frac{1}{2^{n}} \right)\mu_{p}(I).
\end{align*}

\noindent We therefore have that
$$\mu_{p}\left(I \cap \bigcup_{\substack{\by \in I\\
2^{i(n-1)-1} < \height(\by) \leq 2^{i(n-1)}}} B\left(\by, \frac{1}{2^{2i}}\right) \right) \geq \mu_{p}(I) - \left(1 - \frac{1}{2^{n}} \right)\mu_{p}(I) \geq \frac{\mu_{p}(I)}{2^{n}}$$
\noindent thereby completing the proof of the proposition.
\end{proof}

\end{document}